\documentclass{amsart}
\usepackage{amsmath,amsthm,amscd,amssymb}
\usepackage{amsmath,amscd,latexsym, amssymb}
\usepackage{tikz}
\newtheorem{thm}{Theorem}[section]

\newtheorem{lem}[thm]{Lemma}
\newtheorem{cor}[thm]{Corollary}
\newtheorem{exa}[thm]{Example}
\newtheorem{rmk}[thm]{Remark}

\title{On the dot product graph of a commutative ring II}

\author{Mohammad Abdulla and Ayman Badawi}
\address{{\bf Mohammad Abdulla} and {\bf Ayman Badawi}\ \ \mbox{(Corresponding Author)} \\Department of Mathematics and Statistics\\ The American University of Sharjah, P.O. Box 26666 \\ Sharjah, United Arab Emirates \\e-mail: abadawi@aus.edu \ \ \mbox{(Ayman Badawi)}}

\begin{document}

\begin{abstract}
In 2015, the second-named author introduced the {\it dot product graph} associated to a commutative ring $A$. Let $A$ be a commutative ring with nonzero identity, $1 \leq  n  < \infty$  be an integer,  and $R = A \times A \times \cdots \times A$ ($n$ times).  We recall that the {\it total dot product graph} of  $R$ is the (undirected) graph $TD(R)$ with vertices $R^* = R\setminus
\{(0, 0, ..., 0)\}$, and two distinct vertices $x$ and $y$ are adjacent if and
only if $x\cdot y = 0 \in A$ (where $x\cdot y$ denotes the normal dot product of $x$ and $y$). Let $Z(R)$ denote the set of all zero-divisors of $R$. Then the {\it zero-divisor dot product graph} of $R$ is the induced subgraph $ZD(R)$ of  $TD(R)$  with vertices $Z(R)^* = Z(R) \setminus \{(0, 0, ..., 0)\}$.  Let  $U(R)$ denote the set of all units of $R$. Then the {\it unit  dot product graph} of $R$ is the induced subgraph $UD(R)$ of  $TD(R)$  with vertices $U(R)$. In this paper, we study the structure  of $TD(R)$, $UD(R)$, and $ZD(R)$ when $A = Z_n$ or $A = GF(p^n)$, the finite field with $p^n$ elements, where $n \geq 2$ and $p$ is a prime positive integer.
\\[+2mm]
 \subjclassname{ Primary: 13A15; Secondary: 13B99; 05C99}
 \end{abstract}

\maketitle{}

\section{Introduction}
Let  $R$  be a commutative ring with  $1 \not = 0$. Then $Z(R)$ denote the  set of zero-divisors of $R$ and the group of units of  $R$ will
be denoted by  $U(R)$. As usual, $Z_n$ denotes the ring of  integers modulo $n$. The nonzero elements of  $S \subseteq R$
will be denoted by $S^*$. Over the past several years, there has been considerable attention in the literature to associating graphs with commutative rings (and other algebraic structures) and studying the interplay between ring-theoretic and graph-theoretic properties (for example, see [1]--[20] and [22]--[26]). In particular, as in \cite{AL}, the {\it zero-divisor graph} of $R$ is the (simple) graph $\Gamma(R)$ with vertices $Z(R) \setminus \{0\}$, and distinct vertices $x$ and $y$ are adjacent if and only if $xy = 0$. This concept is due to Beck \cite{B}, who let all the elements of $R$ be vertices and was mainly interested in coloring. The zero-divisor  graph of a ring $R$ has been studied extensively by many authors.

In 2015, Badawi \cite{A1} introduced the {\it dot product graph} associated to a commutative ring $A$. Let $A$ be a commutative ring with nonzero identity, $1 \leq  n  < \infty$  be an integer,  and $R = A \times A \times \cdots \times A$ ($n$ times).  We recall from [1] that the {\it total dot product graph} of  $R$ is the (undirected) graph $TD(R)$ with vertices $R^* = R\setminus \{(0, 0, ..., 0)\}$, and two distinct vertices $x$ and $y$ are adjacent if and only if $x\cdot y = 0 \in A$ (where $x\cdot y$ denotes the normal dot product of $x$ and $y$). Let $Z(R)$ denote the set of all zero-divisors of $R$. Then the {\it zero-divisor dot product graph} of $R$ is the induced subgraph $ZD(R)$ of  $TD(R)$  with vertices $Z(R)^* = Z(R) \setminus \{(0, 0, ..., 0)\}$.  Let  $U(R)$ denote the set of all units of $R$. Then the {\it unit  dot product graph} of $R$ is the induced subgraph $UD(R)$ of  $TD(R)$  with vertices $U(R)$. Let $p \geq 2$ be a prime integer, $n \geq 1$, $A = GF(p^n)$ be the finite field with $p^n$ elements, and $R = A\times A$. In section two of this paper, we study  the structure of $ZD(R)$, $UD(R)$, and $TD(R)$. Let $n \geq  2$, $A = Z_n$, and $R = A \times A$. In section three of this paper, we study the structure of $UD(R)$. In section four, we study some induced subgraphs of $ZD(Z_n \times Z_n)$, where $n \geq 2$. In section five, we introduce the {\it equivalence unit dot product} of $R$, $EUD(R)$ and we show that $UD(R)$, can be recovered from $EUD(R)$.

Let  $G$  be a graph. Two vertices $v_1, v_2$ of $G$  are said to be {\it adjacent}  in $G$  if $v_1, v_2$ are connected by an edge (line segment) of $G$ and we write $v_1- v_2$.   A finite sequence  of edges from a vertex $v_1$ of $G$ to a vertex $v_2$ of $G$ is called a {\it path} of $G$ and we write $v_1 - a_1 - a_2 - \cdots - a_k - v_2$, where $k < \infty $ and the $a_i$, $1 \leq i \leq k$,  are some distinct vertices of $G$. Hence it is clear that every edge of $G$ is a path of $G$, but not every path of $G$ is an edge of $G$. We say that  $G$ is {\it connected} if there is a path between any two distinct vertices of  $G$. At the
other extreme, we say that  $G$  is {\it totally disconnected} if no two vertices of  $G$  are adjacent.  We denote the
complete graph on  $n$ vertices by  $K_n$, recall that a graph $G$ is called complete if every two vertices of $G$ are adjacent and the complete
bipartite graph on $m$ and $n$  vertices by $K_{m,n}$ (recall that $K_{m, n}$ is the graph with two sets of vertices, say $V_1, V_2$, such that $|V_1| = n, |V_2| = m$, $V_1 \cap V_2 = \emptyset$, every two vertices in $V_1$ are not adjacent, every two vertices in $V_2$ are not adjacent, and every vertex in $V_1$ is adjacent to every vertex in $V_2$). We will sometimes call a
$K_{1,n}$  a  {\it star graph}. We say that two (induced)
subgraphs $G_1$ and $G_2$ of $G$ are {\it disjoint} if $G_1$ and
$G_2$ have no common vertices and no vertex of $G_1$ (resp.,
$G_2$) is adjacent  (in $G$) to any  vertex not in  $G_1$ (resp.,
$G_2$). Assume that a graph $G = G_1 \cup G_2 \cup \cdots \cup G_n$, where each vertex of $G_i$ is not connected to a vertex of $G_j$ for every $1\leq i, j \leq n$ with $i \not = j$. Then we say that $G$ is the {\it disjoint union} of $G_1, ... , G_n$.

\section{The structure of  $UD(R = A\times A)$ when  $A$ is a field}
Let  $p$ be a positive prime number, $n \geq 2$. Then $A = GF(p^n)$ denotes the finite field with $p^n$ elements. Let $R = A \times A$. Then  $TD(R)$ is not connected by \cite[Theorem 2.1]{A1}. The first two results give a complete description of the structure of $UD(R)$ and $TD(R)$.
\begin{thm}\label{t1} Let  $n \geq 1$, $m = 2^n - 1$  and  $R = GF(2^n) \times GF(2^n)$. Then
\begin{enumerate}
\item $ZD(R) = \Gamma(R) = K_{m, m}$.
\item $UD(R)$ is  the disjoint union of one  $K_m$ and $(2^{(n - 1)} - 1) $  \ $K_{m, m}$'s.
\item $TD(R)$ is the disjoint  union of one  $K_m$ and $2^{(n - 1)}$  \ $K_{m, m}$'s.
  \end{enumerate}
    \end{thm}
    \begin{proof}
    {\bf (1)}. The result is clear by \cite[Theorem 2.1]{A1}, \cite[Theorem 2.1]{AL}, and \cite[Theorem 2.2]{AM}.

    {\bf (2)}. Let $A = GF(2^n)$. Then $R = A \times A$. Let $v_1, v_2  \in U(R)$.  Since $R$ is a vector space over $A$, $v_1 = u(1, a) \in R$ and $v_2 = v(1, b) \in R$ for some  $u, v, a, b \in A^*$. Hence $v_1$ is adjacent to $v_2$ if and only if $v_1 \cdot v_2 = uv + uvab = 0$ in $A$  if and only if $b = -a^{-1} = a^{-1}$ in $A$ .  Thus for each $a \in U(A) = A^*$, let $X_a = \{u(1, a) \mid u \in A^*\}$ and $Y_a = \{u(1, a^{-1}) \mid u \in A^*\}$.  It is clear that $|X_a| = |Y_a| = 2^n - 1$. Let $a = 1$. Since $char(A)= char(R)= 2$, $X_a = Y_a$ and the dot product of every two distinct vertices in $X_a$ is zero. Hence every two distinct  vertices in $X_a$ are adjacent.  Thus the vertices in $X_a$ form the graph $K_m$  that is  a complete subgraph of  $TD(R)$. Let $a \in U(A)$ such that $a \not = 1$. Since  $a^2 \not = 1$ for each $ a \in U(A) \setminus\{1\}$, we have $X_a \cap Y_a = \emptyset$, every two distinct vertices in $X_a$ are not adjacent, and every two distinct vertices in $Y_a$ are not adjacent. Since $char(A) = char(R) = 2$, it is clear that every vertex in $X_a$ is adjacent to every  vertex in $Y_a$. Thus the vertices in $X_a \cup Y_a$ form the graph $K_{m, m}$ that is a complete bi-partite subgraph of $TD(R)$. By construction, there are exactly $(2^n - 2)/ 2 = 2^{n - 1} - 1$ disjoint complete bi-partite $K_{m, m}$ subgraphs of $TD(R)$. Hence $UD(R)$ is the disjoint union  of one complete subgraph $K_m$ and $(2^{n - 1} - 1 )$   complete bi-partite $K_{m, m}$ subgraphs.

   {\bf (3)}. The claim follows from (1) and (2).

    \end{proof}

    \begin{thm}\label{t2}  Let $p \geq 3$ be a positive prime integer, $n \geq 1$, $m = p^n - 1$, and let  $R = GF(p^n) \times GF(p^n)$. Then
\begin{enumerate}
\item $ZD(R) = \Gamma(R) = K_{m, m}$.
\item If $4 \nmid m$, then $UD(R)$ is  the disjoint union  of  $m/2$  $K_{m, m}$'s.
    \item If $4 \mid m$, then $UD(R)$ is  the disjoint  union of two $K_m$'s  and $(m - 2)/2$  \ $K_{m, m}$'s.
        \item If $4\nmid m$, then $TD(R)$ is the disjoint union  of $(m +2)/2$    $K_{m, m}$'s.
        \item If $4\mid m$, then $TD(R)$ is the disjoint union of two $K_m$'s and $m/2$  \ $K_{m, m}$'s.
            \end{enumerate}
    \end{thm}
\begin{proof}
{\bf (1)}. The result is clear by  \cite[Theorem 2.1]{A1}, \cite[Theorem 2.1]{AL}, and \cite[Theorem 2.2]{AM}.

    {\bf (2)} Let $A = GF(p^n)$. Then $R = A \times A$. Let $v_1, v_2  \in U(R)$.  Since $R$ is a vector space over $A$, $v_1 = u(1, a) \in R$ and $v_2 = v(1, b) \in R$ for some  $u, v, a, b \in A^*$. Hence $v_1$ is adjacent to $v_2$ if and only if $v_1 \cdot v_2 = uv + uvab = 0$ in $A$  if and only if $b = -a^{-1}$ in $A$. Since $R$ is a vector space over $A$, for each $a \in U(A) = A^*$, let $X_a = \{u(1, a) \mid u \in A^*\}$ and $Y_a = \{u(1, -a^{-1}) \mid u \in A^*\}$.  It is clear that $|X_a| = |Y_a| = m = p^n - 1$. Since $4\nmid m$, $U(A) = A^*$ has no elements of order 4. Thus $a^2 \not = -1$ for each $ a \in U(A)$. Hence $X_a \cap Y_a = \emptyset$; so every two distinct vertices in $X_a$ are not adjacent, and every two distinct vertices in $Y_a$ are not adjacent.  By construction of $X_a$ and $Y_a$, it is clear that every vertex in $X_a$ is adjacent to every  vertex in $Y_a$. Thus the vertices in $X_a \cup Y_a$ form the graph $K_{m, m}$ that is a complete bi-partite subgraph of $TD(R)$. By construction, there are exactly $m/ 2$  disjoint complete bi-partite $K_{m, m}$ subgraphs of $TD(R)$.   Hence $UD(R)$ is the disjoint union  of  $m/2$  $K_{m, m}$'s.

     {\bf (3)}. Note that $|U(A)| = m$.  Since $U(A) = A^*$ is cyclic and $4 \mid m$, $U(A)$ has exactly one subgroup of order $4$. Thus $U(A)$ has exactly two elements of  order 4, say $b, c$.  Since $a \in U(A)$ is of order 4  if and only if $a^2 = -1$, it is clear that $x^2 = -1$ for some $x \in U(A)$ if and only if $x = b, c$. Let $X_b = \{u(1, b) \mid u \in U(A)\}$ and let $X_c = \{u(1, c) \mid u \in U(A)\}$. It is clear that $|X_b| = |X_c| = m$. Let $H = \{b, c\}$.  Then  the dot product of every two distinct vertices in $X_h$ is zero for each $h \in H$.  Thus every two distinct  vertices in $X_h$ are adjacent for every $h \in H$.  Thus for each $h \in H$, the vertices in $X_h$ form the graph $K_m$  that is  a complete subgraph of  $TD(R)$. Let $a \in U(A) \setminus H$, $X_a = \{u(1, a) \mid u \in A^*\}$, and $Y_a = \{u(1, -a^{-1}) \mid u \in A^*\}$. It is clear that $|X_a| = |Ya| = m$.  Since $a \not \in H$, we have  $X_a \cap Y_a = \emptyset$, every two distinct vertices in $X_a$ are not adjacent, and every two distinct vertices in $Y_a$ are not adjacent. By construction, it is clear that every vertex in $X_a$ is adjacent to every  vertex in $Y_a$. Thus the vertices in $X_a \cup Y_a$ form the graph $K_{m, m}$ that is a complete bi-partite subgraph of $TD(R)$. By construction, there are $(m - 2)/2$ disjoint \ $K_{m, m}$ subgraphs. Hence $UD(R)$ is the disjoint union  of two  $K_m$'s  and $(m - 2)/2$  \ $K_{m, m}$'s.

   {\bf (4)}. The claim follows from (1) and (2).

   {\bf (5)}. The claim follows from (1) and (3).

    \end{proof}

    In view of  Theorem \ref{t2}, we have the following corollary.
\begin{cor}
\label{c1} Let $p \geq 3$ be a prime positive integer, and let $R$ = $\mathbb{Z}_p \times \mathbb{Z}_p$. Then
 \begin{enumerate}
\item $ZD(R)\ = \Gamma(R)$= $K_{p-1,p-1}$ .

\item If $4\nmid p - 1$, then $UD(R)$ is the disjoint union  of $(p - 1)/2$   $K_{p-1,p-1}$.
\item If $4\mid p - 1$, then $UD(R)$ is the disjoint union  of two   $K_{p - 1}$'s  and $(p -  3)/2$  \ $K_{p - 1, p - 1}$'s.
\item If $4\nmid p - 1$, then $TD(R)$ is the disjoint union  of $(p + 1)/2$    $K_{p - 1, p - 1}$'s.
        \item If $4\mid p - 1$, then $TD(R)$ is the disjoint union  of two  $K_{p - 1}$'s and $(p - 1)/2$    \ $K_{p - 1, p - 1}$'s.
\end{enumerate}
\end{cor}

\begin{exa}
Let $A =  \frac{Z_2[X]}{(X^2 + X + 1)}$. Then $A$ is a finite field with 4 elements.  Let  $v = X + (X^2 + X + 1) \in A$. Since $(A^*, .)$ is a cyclic group and $A^* = <v>$, we have $A = \{0, v, v^2, v^3 = 1 + (X^2 + X + 1)\}$ . Let  $R = A\times A$. Then the $UD(R)$ is the disjoint union  of  one $K_3$ and one $K_{3, 3}$ by Theorem \ref{t1}(2). The following is the graph of  $UD(R)$.
\vskip0.2in
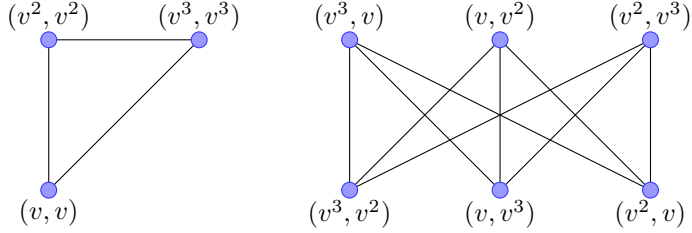
\begin{tikzpicture}[y=.2cm, x=.2cm,font=\sffamily]
\draw (0,0) -- (0,10);
\draw (0,0) -- (10,10);
\draw (0,10) -- (10, 10);
\draw (20, 10) -- (20, 0);
\draw (20, 10) -- (30, 0);
\draw (20, 10) -- (40, 0);
\draw (30, 10) -- (20, 0);
\draw (30, 10) -- (30, 0);
\draw (30, 10)-- (40, 0);
\draw (40, 10) -- (20, 0);
\draw (40, 10) -- (30, 0);
\draw (40, 10) -- (40, 0);
\filldraw[fill=blue!40,draw=blue!80] (0,0) circle (3pt)    node[anchor= north] {$(v, v)$};
\filldraw[fill=blue!40,draw=blue!80] (0,10) circle (3pt)    node[anchor= south] {$(v^2, v^2)$};
\filldraw[fill=blue!40,draw=blue!80] (10, 10) circle (3pt)    node[anchor=south] {$(v^3,  v^3)$};

\filldraw[fill=blue!40,draw=blue!80] (20, 10) circle (3pt)    node[anchor=south] {$(v^3, v)$};
\filldraw[fill=blue!40,draw=blue!80] (30,10) circle (3pt)    node[anchor=south] {$(v, v^2)$};
\filldraw[fill=blue!40,draw=blue!80] (40, 10) circle (3pt)    node[anchor=south] {$(v^2,  v^3)$};

\filldraw[fill=blue!40,draw=blue!80] (20,0) circle (3pt)    node[anchor=north] {$(v^3, v^2)$};
\filldraw[fill=blue!40,draw=blue!80] (30,0) circle (3pt)    node[anchor=north] {$(v, v^3)$};
\filldraw[fill=blue!40,draw=blue!80] (40, 0) circle (3pt)    node[anchor=north] {$(v^2, v)$};
\end{tikzpicture}
\end{exa}
\begin{exa}
Let $A =  Z_5$ and  $R = A\times A$. Then $UD(R)$ is the disjoint union  of  two  $K_4$ and one $K_{4, 4}$ by Corollary \ref{c1}(3). The following is the graph of  $UD(R)$.
\vskip0.2in
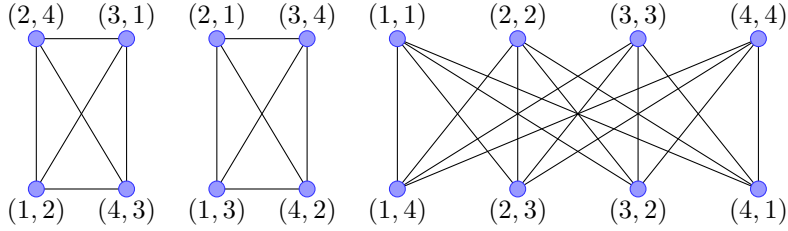
\begin{tikzpicture}[y=.2cm, x=.2cm,font=\sffamily]
\draw (0,0) -- (0,10);
\draw (0,0) -- (6,10);
\draw (0,0) -- (6,0);
\draw (0,10) -- (6,10);
\draw (0,10) -- (6,0);
\draw (6, 0) -- (6,10);
\draw (12,0) -- (12,10);
\draw (12,0) -- (18,10);
\draw (12,0) -- (18,0);
\draw (12,10) -- (18,10);
\draw (12,10) -- (18,0);
\draw (18, 0) -- (18,10);
\draw (24,10) -- (24,0);
\draw (24,10) -- (32,0);
\draw (24,10) -- (40,0);
\draw (24,10) -- (48,0);
\draw (32,10) -- (24,0);
\draw (32,10) -- (32,0);
\draw (32,10) -- (40,0);
\draw (32,10) -- (48,0);
\draw (40,10) -- (24,0);
\draw (40,10) -- (32,0);
\draw (40,10) -- (40,0);
\draw (40,10) -- (48,0);
\draw (48,10) -- (24,0);
\draw (48,10) -- (32,0);
\draw (48,10) -- (40,0);
\draw (48,10) -- (48,0);
\filldraw[fill=blue!40,draw=blue!80] (0,0) circle (3pt)    node[anchor=north] {$(1, 2)$};
\filldraw[fill=blue!40,draw=blue!80] (0,10) circle (3pt)    node[anchor=south] {$(2, 4)$};
\filldraw[fill=blue!40,draw=blue!80] (6,10) circle (3pt)    node[anchor=south] {$(3, 1)$};
\filldraw[fill=blue!40,draw=blue!80] (6,0) circle (3pt)    node[anchor=north] {$(4, 3)$};

\filldraw[fill=blue!40,draw=blue!80] (12,0) circle (3pt)    node[anchor=north] {$(1, 3)$};
\filldraw[fill=blue!40,draw=blue!80] (12,10) circle (3pt)    node[anchor=south] {$(2, 1)$};
\filldraw[fill=blue!40,draw=blue!80] (18,10) circle (3pt)    node[anchor=south] {$(3, 4)$};
\filldraw[fill=blue!40,draw=blue!80] (18,0) circle (3pt)    node[anchor=north] {$(4, 2)$};

\filldraw[fill=blue!40,draw=blue!80] (24,10) circle (3pt)    node[anchor=south] {$(1, 1)$};
\filldraw[fill=blue!40,draw=blue!80] (32,10) circle (3pt)    node[anchor=south] {$(2, 2)$};
\filldraw[fill=blue!40,draw=blue!80] (40,10) circle (3pt)    node[anchor=south] {$(3, 3)$};
\filldraw[fill=blue!40,draw=blue!80] (48,10) circle (3pt)    node[anchor=south] {$(4, 4)$};

\filldraw[fill=blue!40,draw=blue!80] (24,0) circle (3pt)    node[anchor=north] {$(1, 4)$};
\filldraw[fill=blue!40,draw=blue!80] (32,0) circle (3pt)    node[anchor=north] {$(2, 3)$};
\filldraw[fill=blue!40,draw=blue!80] (40,0) circle (3pt)    node[anchor=north] {$(3, 2)$};
\filldraw[fill=blue!40,draw=blue!80] (48,0) circle (3pt)    node[anchor=north] {$(4, 1)$};
\end{tikzpicture}
\end{exa}

\section{Unit dot product graph of $R = Z_n \times Z_n$}

Let $n >1$ and write  $n = p_1^{k_1} \cdots p_m^{k_m}$, where the $p_i$'s are distinct prime positive integers. Then $U(Z_n) = \{ 1 \le a < n \mid $a is an integer and $gcd(a, n) = 1\}$. It is known that  $U(Z_n)$ is a group under multiplication modulo $n$  and  $|U(Z_n)| = \phi(n) = (p_1 - 1)p_1^{k_1 - 1}(p_2 - 1)p_2^{k_2 - 1} \cdots (p_m - 1)p_m^{k_m - 1}$.

\vskip0.in

The following lemma is needed.

\begin{lem} \label{l1}
Let $n$ be a positive integer and write $n = p_1^{k_1}p_2^{k_2} \cdots p_r^{k_r}$, where the $p_i$'s are distinct prime positive integers. Then
\begin{enumerate}
\item  If  $4 \mid  n$, then $a^2 \not \equiv  n-1$ \ (mod \ n) for each $a \in U(Z_n)$.
\item  If $4\nmid n $, then $x^2 \equiv n - 1$ \ (mod \ n) has a solution in $U(Z_n)$ if and only if  $4 \mid (p_i - 1)$ for each  odd prime factor $p_i$ of $n$. Furthermore, if  $x^2 \equiv n -1 \ (mod \ n)$ has a solution in $U(Z_n)$, then it has exactly $2^{r-1}$ distinct solutions in $U(Z_n)$ if $n$ is even and  it has  exactly $2^r$  distinct solutions in $U(Z_n)$ if $n$ is odd.
        \end{enumerate}
        \end{lem}
       \begin{proof}{\bf (1)}. Suppose that  $4\mid n$. Then  $n \geq 4$. Since  $4 \nmid (n-2)$, $n -1 \not \equiv 1 \ (mod \ 4)$ and  thus $a^2 \not \equiv  n-1$ \ (mod \ n) for each $a \in U(Z_n)$ by \cite[Theorem 5.1]{WL}.

       {\bf (2)}. Suppose that  $4 \nmid n$.  Then $a^2 \equiv  n - 1 \ (mod \ n)$ for some  $a \in U(Z_n)$ if and only if  $a^2  \equiv n - 1 \ (mod \ p_i)$ \  for each odd  prime factor $p_i$ of $n$ by \cite[Theorem 5.1]{WL}.  Thus  $a^2 \equiv  n - 1 \ (mod \ n)$ for some  $a \in U(Z_n)$ if and only if  $(a \ mod \ p_i)^2  \equiv  p_i - 1 \ (mod \ p_i)$ \  for each odd  prime factor $p_i$ of $n$. Since $U(Z_{p_i}) = Z^*_{p_i} = \{1, ..., p_i - 1\}$ for each prime factor $p_i$ of $n$,  we have  $|U(Z_{p_i})| = p_i - 1$.  For each  $x \in  U(Z_{p_i})$, $1 \leq i \leq r$, let $|x|$ denotes the order of $x$  in $U(Z_{p_i})$.  Let $p_i$, $1 \leq i \leq r$,  be an odd prime factor of $n$. Since $|p_i - 1| = 2$ in $U(Z_{p_i})$, $b^2  = p_i - 1$ in $U(Z_{p_i})$ for some $b \in U(Z_{p_i})$  if and only if $|b| = 4$ in $U(Z_{p_i})$. Since $|U(Z_{p_i})| = p_i - 1$, we conclude that $b^2 = p_i - 1$ in $U(Z_{P_i})$ for some $b \in U(Z_{p_i})$ if and only if $4 \mid (p_i - 1)$. Thus $x^2 \equiv n - 1$ \ (mod \ n) has a solution in $U(Z_n)$ if and only if  $4 \mid (p_i - 1)$ for each  odd prime $p_i$ factor of $n$.  Suppose that $x^2 \equiv n -1 \ (mod \ n)$ has a solution in $U(Z_n)$. We consider two cases: {\bf Case 1}.  Suppose that  $n$ is an even integer. Then there are  exactly $r - 1$ distinct odd prime factors of $n$.  Since $4 \nmid n$, $x^2 \equiv n -1 \ (mod \ n)$ has exactly $2^{r-1}$ distinct solutions in $U(Z_n)$ by  \cite[Theorem 5.2]{WL}. {\bf Case 2}. Suppose that $n$ is an odd integer. Then there are  exactly $r$ distinct odd prime factors of $n$. Thus  $x^2 \equiv n -1 \ (mod \ n)$ has exactly $2^r$ distinct solutions in $U(Z_n)$ by  \cite[Theorem 5.2]{WL}.
       \end{proof}
    \vskip0.2in
    Let  $A = Z_n$, where $n$ is not prime. Then  $TD(A\times A)$ is connected  by \cite[Theorem 2.3 ]{A1}. In the following result, we show that $UD(A \times A)$ is disconnected, and we give a complete description of the structure of  $UD(A \times A)$.
  \begin{thm}\label{tt1}  Let $n \geq 3$ be an integer, $R = Z_n \times Z_n$ and $\phi(n) = m $. Write $n = p_1^{k_1}p_2^{k_2} \cdots p_r^{k_r}$, where the $p_i$'s are distinct prime positive integers, $1\leq i \leq r$. Then
  \begin{enumerate}
\item   If  $4 \mid  n$, then $UD(R)$ is  the disjoint union  of  $m/2$  $K_{m, m}$'s.
 \item  If  $4 \nmid  n$ and  $4 \nmid (p_i - 1)$ for at least one of the $p_i$'s in the prime factorization of $n$, then $UD(R)$ is  the disjoint union  of  $m/2$  $K_{m, m}$'s.
 \item  If  $4 \nmid  n$ and  $4 \mid (p_i - 1)$ for all the odd $p_i$'s in the prime factorization of $n$, then we consider two cases:

{\bf Case I}. If $n$ is even, then  $UD(R)$ is the disjoint union  of  $(m/2) - 2^{r-2}$    $K_{m, m}$'s and
  $2^{r-1}$    $K_m$'s.

 {\bf Case II}. If n is odd, then $UD(R)$ is the disjoint union  of  $(m/2) - 2^{r-1}$    $K_{m, m}$'s and
  $2^r$   $K_m$'s.
  \end{enumerate}

    \end{thm}

    \begin{proof}
      Let $A = Z_n$. Then $R = A\times A$. Note that $UD(R)$ has exactly $m^2$  vertices.  Let $v_1, v_2  \in U(R)$.  Since $R$ is a vector space over $A$, $v_1 = u(1, a) \in R$ and $v_2 = v(1, b) \in R$ for some  $u, v, a, b \in U(A)$. Hence $v_1$ is adjacent to $v_2$ if and only if $v_1 \cdot v_2 = uv + uvab = 0$ in $A$  if and only if $b = -a^{-1}$ in $A$.  Thus for each $a \in U(A)$, let $X_a = \{u(1, a) \mid u \in U(A)\}$ and $Y_a = \{u(1, - a^{-1}) \mid u \in U(A)\}$. It is clear that $|X_a| = |Y_a| = m $.

{\bf (1)}. Since $4 \mid  n$,  $a^2 \not \equiv  n-1$ \ (mod \ n) for each $a \in U(Z_n)$ by Lemma \ref{l1}(1). Hence $X_a \cap Y_a = \emptyset$; so every two distinct vertices in $X_a$ are not adjacent, and every two distinct vertices in $Y_a$ are not adjacent.  By construction of $X_a$ and $Y_a$, it is clear that every vertex in $X_a$ is adjacent to every  vertex in $Y_a$. Thus the vertices in $X_a \cup Y_a$ form the graph $K_{m, m}$ that is a complete bi-partite subgraph of $TD(R)$. By construction, there are exactly $m/ 2$  disjoint complete bi-partite $K_{m, m}$ subgraphs of $TD(R)$.   Hence $UD(R)$ is the disjoint union  of  $m/2$  $K_{m, m}$'s.

 {\bf (2)}. Write $n = p_1^{k_1}p_2^{k_2} \cdots p_r^{k_r}$, where the $p_i$'s are distinct prime positive integers. Since $4 \nmid  n$ and $4 \nmid (p_i - 1)$ for at least one of the $p_i$'s, $a^2 \not \equiv  n-1$ \ (mod \ n) for each $a \in U(Z_n)$ by Lemma \ref{l1}. Thus by the same argument as in (1), $UD(R)$ is the disjoint union  of  $m/2$  $K_{m, m}$'s.

{\bf (3)}.  Write $n = p_1^{k_1}p_2^{k_2} \cdots p_r^{k_r}$, where the $p_i$'s are distinct prime positive integers. Suppose that $4 \nmid  n$ and  $4 \mid p_i-1$ for all the odd $p_i$'s in the prime factorization of $n$. Let $B = \{b \in U(Z_n) \mid b^2 = n -1$ in $U(Z_n)\}$ and $
C = \{c \in U(Z_n) \mid c^2 \not = n -1$ in $U(Z_n)\}$.  We consider two cases: {\bf Case I}.  Suppose that $n$ is even. Then $|B| = 2^{r - 1}$ by Lemma \ref{l1}(2) and hence $|C| = m - 2^{r - 1}$. For each $a \in B$, we have $X_a = Y_a$ and hence the dot product of every two distinct vertices in $X_a$ is zero. Thus the vertices in $X_a$ form the graph $K_m$  that is  a complete subgraph of  $TD(R)$. Hence  $UD(Z_n)$ has exactly  $2^{r-1}$  disjoint  $K_m$'s. For each $a \in C$, we have  $X_a \cap Y_a = \emptyset$; so every two distinct vertices in $X_a$ are not adjacent, and every two distinct vertices in $Y_a$ are not adjacent. By construction, it is clear that every vertex in $X_a$ is adjacent to every  vertex in $Y_a$. Thus the vertices in $X_a \cup Y_a$ form the graph $K_{m, m}$ that is a complete bi-partite subgraph of $TD(R)$. Thus $UD(Z_n)$ has exactly $\frac{m - 2^{r-1}}{2} = \frac{m}{2} - 2^{r - 2}$ disjoint $K_{m, m}$'s.  {\bf Case II}. Suppose that $n$ is odd. Then $|B| = 2^r $ by Lemma \ref{l1}(2) and hence $|C| = m - 2^r$. For each $a \in B$, we have $X_a = Y_a$ and hence the dot product of every two distinct vertices in $X_a$ is zero. Thus the vertices in $X_a$ form the graph $K_m$  that is  a complete subgraph of  $TD(R)$. Hence  $UD(Z_n)$ has exactly  $2^r$  disjoint  $K_m$'s. For each $a \in C$, we have  $X_a \cap Y_a = \emptyset$, every two distinct vertices in $X_a$ are not adjacent, and every two distinct vertices in $Y_a$ are not adjacent. By construction, it is clear that every vertex in $X_a$ is adjacent to every  vertex in $Y_a$. Thus the vertices in $X_a \cup Y_a$ form the graph $K_{m, m}$ that is a complete bi-partite subgraph of $TD(R)$. Thus $UD(Z_n)$ has exactly $\frac{m - 2^r}{2} = \frac{m}{2} - 2^{r - 1}$ disjoint $K_{m, m}$'s.
  \end{proof}
\vskip0.2in
Recall that a graph $G$  is called {\it completely disconnected} if every two vertices of $G$ are not connected by an edge in $G$.

\begin{thm}
 Let $n \ge 4$ be an even integer, and let $R = Z_n \times Z_n \times \cdots \times Z_n $  ($k$  times),  where $k$ is an odd positive integer.
Then $UD(R)$ is  completely disconnected.
\end{thm}
\begin{proof}
Let $x = (x_1, ..., x_k), y = (y_1, ..., y_k) \in U(R)$. Then  $x_i, y_i \in U(Z_n)$ for every i, $1 \leq i \leq k$.  Since $n$ is an even integer, $x_i$ and $y_i$ are odd integers for every $i$, $1 \leq i \leq k$.  Hence, since $k$ is an odd integer, $x \cdot y = x_1y_1 + \cdots + x_ky_k$ is an odd integer, and thus $x \cdot y = x_1y_1 + \cdots + x_ky_k \not = 0$ in $Z_n$, since $n$ is even. Thus  $UD(R)$ is  completely disconnected.
\end{proof}

\begin{thm}
 Let $n \geq 4$ be an even integer, and let $R = Z_n \times Z_n$. Then  the vertex $(n/2,n/2)$ in $ZD(R)$ is adjacent to every vertex in $UD(R)$.
\end{thm}
\begin{proof}
It is clear that  $(\frac{n}{2}, \frac{n}{2})$  is a vertex of $ZD(R)$. Let $u \in U(Z_n)$. Since $n$ is even, $u$ is an odd integer. Thus $u - 1 = 2m$ for some  integer $m$.  Hence  $\frac{n}{2}(u - 1) = \frac{n}{2}(2m) = mn = 0 \in Z_n$. Thus $\frac{n}{2}u = \frac{n}{2}$ in $Z_n$. Now let $(a, b) \in U(R)$. Then $a, b \in U(Z_n)$ are odd integers. Hence $(a, b)(\frac{n}{2}, \frac{n}{2}) = \frac{n}{2} + \frac{n}{2} = n = 0 \in Z_n$. Thus the  vertex $(n/2,n/2)$ in $ZD(R)$ is adjacent to every vertex in $UD(R)$.
\end{proof}

\begin{exa}
Let $A = Z_8$ and  $R = A\times A$. Then $UD(R)$ is the disjoint union  of  two  $K_{4, 4}$ by Theorem  \ref{tt1}(1). The following is the graph of  $UD(R)$.
\vskip0.2in
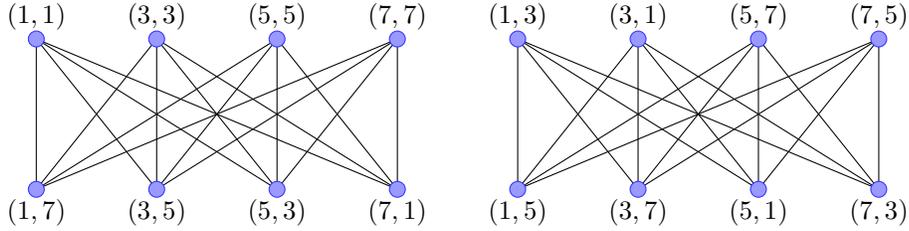
\begin{tikzpicture}[y=.2cm, x=.2cm,font=\sffamily]
\draw (0,0) -- (0,10);
\draw (0,0) -- (8,10);
\draw (0,0) -- (16, 10);
\draw (0,0)-- (24,10);

\draw (8,0) -- (0,10);
\draw (8,0) -- (8,10);
\draw (8,0) -- (16, 10);
\draw (8,0)-- (24,10);

\draw (16,0) -- (0,10);
\draw (16,0) -- (8,10);
\draw (16,0) -- (16, 10);
\draw (16,0)-- (24,10);

\draw (24,0) -- (0,10);
\draw (24,0) -- (8,10);
\draw (24,0) -- (16, 10);
\draw (24,0)-- (24,10);

\draw (32,0) -- (32,10);
\draw (32,0) -- (40,10);
\draw (32,0) -- (48, 10);
\draw (32,0)-- (56,10);

\draw (40,0) -- (32,10);
\draw (40,0) -- (40,10);
\draw (40,0) -- (48, 10);
\draw (40,0)-- (56,10);

\draw (48,0) -- (32,10);
\draw (48,0) -- (40,10);
\draw (48,0) -- (48, 10);
\draw (48,0)-- (56,10);

\draw (56,0) -- (32,10);
\draw (56,0) -- (40,10);
\draw (56,0) -- (48, 10);
\draw (56,0)-- (56,10);

\filldraw[fill=blue!40,draw=blue!80] (0, 10) circle (3pt)    node[anchor=south] {$(1, 1)$};
\filldraw[fill=blue!40,draw=blue!80] (8,10) circle (3pt)    node[anchor=south] {$(3, 3)$};
\filldraw[fill=blue!40,draw=blue!80] (16, 10) circle (3pt)    node[anchor=south] {$(5, 5)$};
\filldraw[fill=blue!40,draw=blue!80] (24, 10) circle (3pt)    node[anchor=south] {$(7, 7)$};
\filldraw[fill=blue!40,draw=blue!80] (32,10) circle (3pt)    node[anchor=south] {$(1, 3)$};
\filldraw[fill=blue!40,draw=blue!80] (40, 10) circle (3pt)    node[anchor=south] {$(3, 1)$};
\filldraw[fill=blue!40,draw=blue!80] (48, 10) circle (3pt)    node[anchor=south] {$(5, 7)$};
\filldraw[fill=blue!40,draw=blue!80] (56, 10) circle (3pt)    node[anchor=south] {$(7,  5)$};

\filldraw[fill=blue!40,draw=blue!80] (0,0) circle (3pt)    node[anchor=north] {$(1, 7)$};
\filldraw[fill=blue!40,draw=blue!80] (8,0) circle (3pt)    node[anchor=north] {$(3, 5)$};
\filldraw[fill=blue!40,draw=blue!80] (16, 0) circle (3pt)    node[anchor=north] {$(5, 3)$};
\filldraw[fill=blue!40,draw=blue!80] (24,0) circle (3pt)    node[anchor=north] {$(7, 1)$};
\filldraw[fill=blue!40,draw=blue!80] (32,0) circle (3pt)    node[anchor=north] {$(1, 5)$};
\filldraw[fill=blue!40,draw=blue!80] (40, 0) circle (3pt)    node[anchor=north] {$(3, 7)$};
\filldraw[fill=blue!40,draw=blue!80] (48,0) circle (3pt)    node[anchor=north] {$(5, 1)$};
\filldraw[fill=blue!40,draw=blue!80] (56, 0) circle (3pt)    node[anchor=north] {$(7, 3)$};
\end{tikzpicture}
\end{exa}

\begin{exa}
Let $A = Z_{10}$ and  $R = A\times A$. Then  $UD(R)$ is the disjoint union  of  two  $K_4$ and one  $K_{4, 4}$ by Theorem  \ref{tt1}(3, case I). The following is the graph of  $UD(R)$.
\vskip0.2in
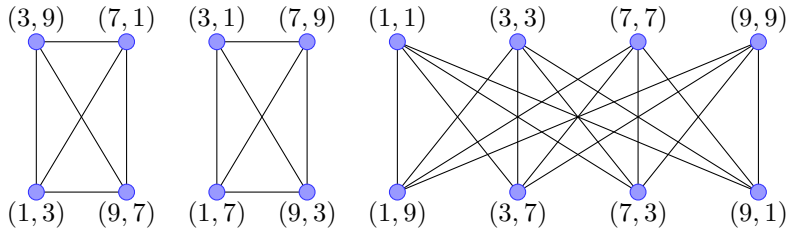
\begin{tikzpicture}[y=.2cm, x=.2cm,font=\sffamily]
\draw (0,0) -- (0,10);
\draw (0,0) -- (6,10);
\draw (0,0) -- (6,0);
\draw (0,10) -- (6,10);
\draw (0,10) -- (6,0);
\draw (6, 0) -- (6,10);

\draw (12,0) -- (12,10);
\draw (12,0) -- (18,10);
\draw (12,0) -- (18,0);
\draw (12,10) -- (18,10);
\draw (12,10) -- (18,0);
\draw (18, 0) -- (18,10);

\draw (24,10) -- (24,0);
\draw (24,10) -- (32,0);
\draw (24,10) -- (40,0);
\draw (24,10) -- (48,0);

\draw (32,10) -- (24,0);
\draw (32,10) -- (32,0);
\draw (32,10) -- (40,0);
\draw (32,10) -- (48,0);

\draw (40,10) -- (24,0);
\draw (40,10) -- (32,0);
\draw (40,10) -- (40,0);
\draw (40,10) -- (48,0);

\draw (48,10) -- (24,0);
\draw (48,10) -- (32,0);
\draw (48,10) -- (40,0);
\draw (48,10) -- (48,0);
\filldraw[fill=blue!40,draw=blue!80] (0,0) circle (3pt)    node[anchor=north] {$(1, 3)$};
\filldraw[fill=blue!40,draw=blue!80] (0,10) circle (3pt)    node[anchor=south] {$(3, 9)$};
\filldraw[fill=blue!40,draw=blue!80] (6,10) circle (3pt)    node[anchor=south] {$(7, 1)$};
\filldraw[fill=blue!40,draw=blue!80] (6,0) circle (3pt)    node[anchor=north] {$(9, 7)$};

\filldraw[fill=blue!40,draw=blue!80] (12,0) circle (3pt)    node[anchor=north] {$(1, 7)$};
\filldraw[fill=blue!40,draw=blue!80] (12,10) circle (3pt)    node[anchor=south] {$(3, 1)$};
\filldraw[fill=blue!40,draw=blue!80] (18,10) circle (3pt)    node[anchor=south] {$(7, 9)$};
\filldraw[fill=blue!40,draw=blue!80] (18,0) circle (3pt)    node[anchor=north] {$(9, 3)$};

\filldraw[fill=blue!40,draw=blue!80] (24,10) circle (3pt)    node[anchor=south] {$(1, 1)$};
\filldraw[fill=blue!40,draw=blue!80] (32,10) circle (3pt)    node[anchor=south] {$(3, 3)$};
\filldraw[fill=blue!40,draw=blue!80] (40,10) circle (3pt)    node[anchor=south] {$(7, 7)$};
\filldraw[fill=blue!40,draw=blue!80] (48,10) circle (3pt)    node[anchor=south] {$(9, 9)$};

\filldraw[fill=blue!40,draw=blue!80] (24,0) circle (3pt)    node[anchor=north] {$(1, 9)$};
\filldraw[fill=blue!40,draw=blue!80] (32,0) circle (3pt)    node[anchor=north] {$(3, 7)$};
\filldraw[fill=blue!40,draw=blue!80] (40,0) circle (3pt)    node[anchor=north] {$(7, 3)$};
\filldraw[fill=blue!40,draw=blue!80] (48,0) circle (3pt)    node[anchor=north] {$(9, 1)$};
\end{tikzpicture}
\end{exa}

\section{Subgraphs of the zero-divisor dot product graph of $Z_n \times Z_n$}\label{s4}

 For an integer  $n \geq 2$, let  $R_1=\{(u_1,z_1)\mid u_1 \in U(Z_n)$\ and  $z_1 \in Z(Z_n)\}$ \ and  $R_2 = \{(z_2 ,u_2)\mid u_2 \in U(Z_n)$ \ and  $z_2 \in Z(Z_n)\}$. It is clear that  $R_1 \subset Z(Z_n \times Z_n)$ and $R_2 \subset Z(Z_n \times Z_n)$. In this section, we study the induced subgraph $ZD(R_1\cup R_2)$ of $ZD(Z_n \times Z_n)$  with  vertices  $R_1 \cup R_2$.

 \begin{thm}\label{ttt1}  Let  $n \geq 2$, $R = Z_n \times Z_n$, and $\phi(n)= m $. Then
 \begin{enumerate}
 \item If $n$ is prime, then $ZD(R_1\cup R_2)$ =  $ZD(Z_n \times Z_n)$ = $\Gamma(R) = K_{n - 1, n - 1}$.
  \item If $n$ is not prime, then $ZD(R_1\cup R_2)$ is  the disjoint union  of of $(n - m)$    $K_{m, m}$'s.
  \end{enumerate}
\end{thm}

 \begin{proof}
{\bf (1)}. Suppose that  $n$ is prime. Then it is clear that $R_1 \cup R_2 = Z(Z_n \times Z_n)$. If $n = 2$, then it is trivial to see that $ZD(R_1\cup R_2)$ =  $ZD(Z_n \times Z_n)$ = $\Gamma(R) = K_{1, 1}$. If $n \geq 3$, then  the claim is clear by Corollary \ref{c1}(1).

{\bf (2)}.  Let $A = Z_n$. Suppose that $n$ is not prime.  It is clear that every two vertices in $R_i$ are not adjacent for every $i \in \{1, 2\}$.  Let $v_1 \in R_1$ and $v_2 \in R_2$. Then $v_1 = u(1, a) \in R_1$ and $v_2 = v(b, 1) \in R_2$ for some  $u, v \in U(A)$ and some $a, b \in Z(A)$. Then  $v_1$ is adjacent to $v_2$ if and only if $v_1 \cdot v_2 = uvb + uva = 0$ in $A$  if and only if $b = -a$ in $A$.  Hence for each $a \in Z(A)$, let $X_a = \{u(1, a) \mid u \in U(A)\}$ and $Y_a = \{u(-a, 1) \mid u \in U(A)\}$. It is clear that $|X_a| = |Y_a| = m $. For each $a \in Z(A)$, $X_a \cap Y_a = \emptyset$, every two distinct vertices in $X_a$ are not adjacent, and every two distinct vertices in $Y_a$ are not adjacent. By construction, it is clear that every vertex in $X_a$ is adjacent to every  vertex in $Y_a$. Thus the vertices in $X_a \cup Y_a$ form the graph $K_{m, m}$ that is a complete bi-partite subgraph of $ZD(R)$. Since $|R_1| = |R_2| =  m(n - m)$ and  $R_1 \cap R_2 = \emptyset$, we have $|R_1 \cup R_2| = 2m(n - m)$. Thus  $ZD(R_1\cup R_2)$ is  the disjoint union  of of $(n - m)$    $K_{m, m}$'s.
\end{proof}

\section{equivalence dot product graph}
Let $A = Z_n$ and $R = A \times A$. Define a relation $\thicksim$ on $U(R)$ such that  $x \thicksim y$, where $x, y \in U(R)$, if $x  = (c, c)y$ for some  $(c, c) \in U(R)$. It is clear that $\thicksim$ is an equivalence relation on $U(R)$. If $S$ is  an equivalence class of $U(R)$, then there is an $a \in U(A)$ such that $S  = \overline{(1, a)}= \{u(1, a) \mid u \in U(Z_n) \}$. Let $E(U(R))$ be the set of all distinct equivalence classes of  $U(R)$.  We  define the {\it equivalence unit  dot product graph} of $U(R)$ to be  the (undirected) graph $EUD(R)$ with vertices $E(U(R))$, and two distinct vertices $X$ and $Y$ are adjacent if and only if $a\cdot b = 0 \in A$ for every $a \in X$  and every $b \in Y$ (where $a\cdot b$ denote the normal dot product of $a$ and $b$). We have the following results.

\begin{thm}\label{rt1} Let  $n \geq 1$, $m = 2^n - 1$  and  $R = GF(2^n) \times GF(2^n)$. Then $EUD(R)$ is  the disjoint union  of one  $K_1$ and $(2^{(n - 1)} - 1) $  \ $K_{1, 1}$'s.

    \end{thm}
    \begin{proof} Let $A = GF(2^n)$. For each  $a \in U(A)$, let $X_a$ and $Y_a$ be as in the proof of Theorem \ref{t1}. Then $X_a, Y_a \in E(U(R))$. Since $|X| = m$ for each $X \in E(U(R))$,  we conclude that each $K_m$ of  $UD(R)$ is a $K_1$ of $EUD(R)$ and  each  $K_{m, m}$ of  $UD(R)$ is a $K_{1, 1}$ of $EUD(R)$. Hence the claim follows by the proof of  Theorem \ref{t1}.
\end{proof}

 \begin{thm}\label{rt2}  Let $p \geq 3$ be a positive prime integer, $n \geq 1$, $m = p^n - 1$, and let  $R = GF(p^n) \times GF(p^n)$. Then
\begin{enumerate}
\item If $4 \nmid m$, then $EUD(R)$ is  the disjoint union  of  $m/2$  $K_{1, 1}$'s.
    \item If $4 \mid m$, then $EUD(R)$ is  the disjoint union  of two  $K_1$'s  and $(m - 2)/2$  \ $K_{1, 1}$'s.

            \end{enumerate}
    \end{thm}
\begin{proof}
Let $A = GF(p^n)$. For each  $a \in U(A)$, let $X_a$ and $Y_a$ be as in the proof of Theorem \ref{t2}. Then $X_a, Y_a \in E(U(R))$. Since $|X| = m$ for each $X \in E(U(R))$,  we conclude that each $K_m$ of  $UD(R)$ is a $K_1$ of $EUD(R)$ and  each  $K_{m, m}$ of  $UD(R)$ is a $K_{1, 1}$ of $EUD(R)$. Hence the claim follows by the proof of  Theorem \ref{t2}.
\end{proof}

\begin{thm}\label{rtt1}  Let $n \geq 3$ be an integer, $R = Z_n \times Z_n$ and $\phi(n)=m $. Write $n = p_1^{k_1}p_2^{k_2} \cdots p_r^{k_r}$, where the $p_i$'s are distinct prime positive integers, $1\leq i \leq r$. Then
  \begin{enumerate}
\item   If  $4 \mid  n$, then $EUD(R)$ is  the disjoint union  of  $m/2$  $K_{1, 1}$'s.
 \item  If  $4 \nmid  n$ and  $4 \nmid (p_i - 1)$ for at least one of the $p_i$'s in the prime factorization of $n$, then $EUD(R)$ is  the disjoint union  of  $m/2$  $K_{1, 1}$'s.
 \item  If  $4 \nmid  n$ and  $4 \mid (p_i - 1)$ for all the odd $p_i$'s in the prime factorization of $n$, then we consider two cases:

{\bf Case I}. If $n$ is even, then  $EUD(R)$ is the disjoint union  of  $(m/2) - 2^{r-2}$    $K_{1, 1}$'s and
  $2^{r-1}$    $K_1$'s.

 {\bf Case II}. If n is odd, then $EUD(R)$ is the disjoint union  of  $(m/2) - 2^{r-1}$    $K_{1, 1}$'s and
  $2^r$    $K_1$'s.
  \end{enumerate}
  \end{thm}
\begin{proof}
Let $A = Z_n$. For each  $a \in U(A)$, let $X_a$ and $Y_a$ be as in the proof of Theorem \ref{tt1}. Then $X_a, Y_a \in E(U(R))$. Since $|X| = m$ for each $X \in E(U(R))$,  we conclude  that each $K_m$ of  $UD(R)$ is a $K_1$ of $EUD(R)$ and  each  $K_{m, m}$ of  $UD(R)$ is a $K_{1, 1}$ of $EUD(R)$. Hence the claim follows by the proof of  Theorem \ref{tt1}.
    \end{proof}

Let $R_1=\{(u_1,z_1)\mid u_1 \in U(Z_n)$\ and  $z_1 \in Z(Z_n)\}$ \ and  $R_2 = \{(z_2 ,u_2)\mid u_2 \in U(Z_n)$ \ and  $z_2 \in Z(Z_n)\}$, see section \ref{s4}.  We define a relation $\thicksim$ on $R_1 \cup R_2$ such that  $x \thicksim y$, where $x, y \in R_1 \cup R_2$, if $x  = (c, c)y$ for some  $(c, c) \in U(Z_n \times Z_n)$. It is clear that $\thicksim$ is an equivalence relation on $R_1 \cup R_2$. By construction of $R_1$ and $R_2$, it is clear that if $x \thicksim y$ for some $x, y \in R_1 \cup R_2$, then $x, y \in R_1$ or  $x, y \in R_2$. Hence if  $S$ is  an equivalence class of $R_1 \cup R_2$, then there is an $a \in Z(Z_n)$ such that either $S  = (\overline{(1, a)} = \{u(1, a) \mid u \in U(Z_n) \}$ or  $S =  \overline{(a, 1)} = \{u(a, 1) \mid u \in U(Z_n) \}$. Let $E(R_1 \cup R_2)$ be the set of all distinct equivalence classes of  $R_1 \cup R_2$.  We  define the {\it equivalence zero-divisor dot product graph} $R_1 \cup R_2$ to be  the (undirected) graph $EZD(R_1 \cup R_2)$ with vertices $E(R_1 \cup R_2)$, and two distinct vertices $X$ and $Y$ are adjacent if and only if $a\cdot b = 0 \in A$ for every $a \in X$  and every $b \in Y$ (where $a\cdot b$ denote the normal dot product of $a$ and $b$). We have the following result.

\begin{thm}\label{rttt1}  Let $n \geq 2, R = Z_n \times Z_n$, and $\phi(n)= m $. Then
 \begin{enumerate}
 \item If $n$ is prime, then $EZD(R_1\cup R_2) = K_{1, 1}$.
  \item If $n$ is not prime, then $EZD(R_1\cup R_2)$ is the disjoint union  of of $(n - m)$    $K_{1, 1}$'s.
  \end{enumerate}
\end{thm}
 \begin{proof}
{\bf (1)}. If $n$ is prime, then $E = \{\overline{(1, 0)}, \overline{(0, 1)}\}$. Thus  $EZD(R_1\cup R_2) = K_{1, 1}$.

{\bf (2)}. Suppose that $n$ is not prime, and let $A = Z_n$. For each  $a \in Z(A)$, let $X_a$ and $Y_a$ be as in the proof of Theorem \ref{ttt1}. Then $X_a, Y_a \in E(R_1 \cup R_2)$. Since $|X| = m$ for each $X \in E(R_1 \cup R_2)$,  we conclude that each $K_{m, m}$ of  $ZD(R_1\cup R_2)$ is a $K_{1, 1}$ of $EZD(R_1\cup R_2)$. Hence the claim follows by the proof of  Theorem \ref{ttt1}.
\end{proof}

\begin{rmk}
\begin{enumerate}
\item Let $A = Z_n$ and  $R = Z_n \times Z_n$. Since for each $X \in E(U(R))$ there exists an $a \in U(A)$ such that  $X = \overline{(1, a)} = \{u(1, a) \mid u \in U(A)\}$, note that we can recover the graph $UD(R)$ from the graph $EUD(R)$. However, drawing $EUD(R)$ is much simpler than drawing $UD(R)$.
\item Since for each $X \in E(R_1 \cup R_2)$ there exists an $a \in Z(Z_n)$ such that either $X = \overline{(1, a)} = \{u(1, a) \mid u \in U(Z_n) \}$ or  $X  = \overline{(a, 1)} = \{u(a, 1) \mid u \in U(Z_n) \}$, note that we can recover the graph $ZD(R_1 \cup R_2)$ from the graph $EZD(R_1 \cup R_2)$. However, drawing $EZD(R_1 \cup R_2)$ is much simpler than drawing $ZD(R_1 \cup R_2)$.
    \end{enumerate}
\end{rmk}
\begin{exa}
Let $A = Z_{20}$ and  $R = A\times A$. Then $EUD(R)$ is the disjoint union  of 4  $K_{1, 1}$ by Theorem  \ref{rtt1}(1), and thus  $UD(R)$ is the disjoint union  of 4  $K_{8, 8}$. The following is the graph of  $EUD(R)$.
\vskip0.2in
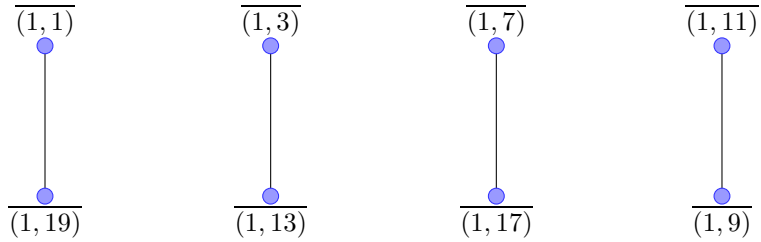
\begin{tikzpicture}[y=.2cm, x=.2cm,font=\sffamily]
\draw (0,0) -- (0,10);
\draw (15,0) -- (15,10);
\draw (30,0) -- (30, 10);
\draw (45,0)-- (45,10);

\filldraw[fill=blue!40,draw=blue!80] (0, 10) circle (3pt)    node[anchor=south] {$\overline{(1, 1)}$};
\filldraw[fill=blue!40,draw=blue!80] (15,10) circle (3pt)    node[anchor=south] {$\overline{(1, 3)}$};
\filldraw[fill=blue!40,draw=blue!80] (30, 10) circle (3pt)    node[anchor=south] {$\overline{(1, 7)}$};
\filldraw[fill=blue!40,draw=blue!80] (45,10) circle (3pt)    node[anchor=south] {$\overline{(1, 11)}$};

\filldraw[fill=blue!40,draw=blue!80] (0,0) circle (3pt)    node[anchor=north] {$\overline{(1, 19)}$};
\filldraw[fill=blue!40,draw=blue!80] (15, 0) circle (3pt)    node[anchor=north] {$\overline{(1, 13)}$};
\filldraw[fill=blue!40,draw=blue!80] (30,0) circle (3pt)    node[anchor=north] {$\overline{(1, 17)}$};
\filldraw[fill=blue!40,draw=blue!80] (45, 0) circle (3pt)    node[anchor=north] {$\overline{(1, 9)}$};
\end{tikzpicture}
\end{exa}
\begin{exa}
Let $A = Z_{34}$ and  $R = A\times A$. Then $EUD(R)$ is the disjoint union  of 7  $K_{1, 1}$'s  and  2  $K_1$'s by Theorem  \ref{rtt1}(3, Case I), and thus  $UD(R)$ is the disjoint union  of 7 $K_{16, 16}$ and 2  $K_8$. The following is the graph of  $EUD(R)$.
\vskip0.2in
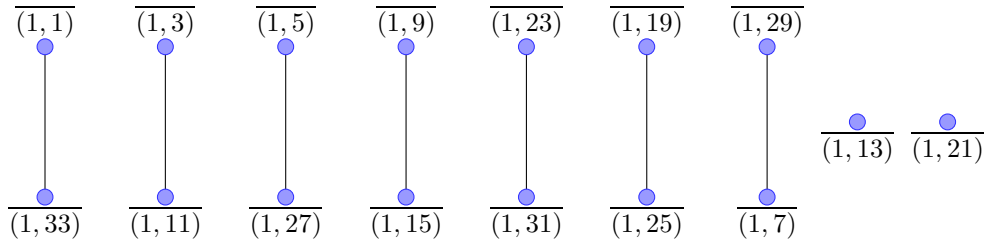
\begin{tikzpicture}[y=.2cm, x=.2cm,font=\sffamily]
\draw (0,0) -- (0,10);
\draw (8,0) -- (8,10);
\draw (16,0) -- (16, 10);
\draw (24,0)-- (24,10);
\draw (32,0) -- (32,10);
\draw (40,0) -- (40, 10);
\draw (48,0)-- (48,10);

\filldraw[fill=blue!40,draw=blue!80] (0, 10) circle (3pt)    node[anchor=south] {$\overline{(1, 1)}$};
\filldraw[fill=blue!40,draw=blue!80] (8,10) circle (3pt)    node[anchor=south] {$\overline{(1, 3)}$};
\filldraw[fill=blue!40,draw=blue!80] (16, 10) circle (3pt)    node[anchor=south] {$\overline{(1, 5)}$};
\filldraw[fill=blue!40,draw=blue!80] (24,10) circle (3pt)    node[anchor=south] {$\overline{(1, 9)}$};
\filldraw[fill=blue!40,draw=blue!80] (32,10) circle (3pt)    node[anchor=south] {$\overline{(1, 23)}$};
\filldraw[fill=blue!40,draw=blue!80] (40, 10) circle (3pt)    node[anchor=south] {$\overline{(1,19 )}$};
\filldraw[fill=blue!40,draw=blue!80] (48,10) circle (3pt)    node[anchor=south] {$\overline{(1, 29)}$};

\filldraw[fill=blue!40,draw=blue!80] (0,0) circle (3pt)    node[anchor=north] {$\overline{(1, 33)}$};
\filldraw[fill=blue!40,draw=blue!80] (8, 0) circle (3pt)    node[anchor=north] {$\overline{(1, 11)}$};
\filldraw[fill=blue!40,draw=blue!80] (16,0) circle (3pt)    node[anchor=north] {$\overline{(1, 27)}$};
\filldraw[fill=blue!40,draw=blue!80] (24, 0) circle (3pt)    node[anchor=north] {$\overline{(1, 15)}$};
\filldraw[fill=blue!40,draw=blue!80] (32, 0) circle (3pt)    node[anchor=north] {$\overline{(1, 31 )}$};
\filldraw[fill=blue!40,draw=blue!80] (40,0) circle (3pt)    node[anchor=north] {$\overline{(1, 25)}$};
\filldraw[fill=blue!40,draw=blue!80] (48, 0) circle (3pt)    node[anchor=north] {$\overline{(1, 7)}$};
\filldraw[fill=blue!40,draw=blue!80] (54,5) circle (3pt)    node[anchor=north] {$\overline{(1, 13)}$};
\filldraw[fill=blue!40,draw=blue!80] (60,5) circle (3pt)    node[anchor=north] {$\overline{(1, 21)}$};

\end{tikzpicture}
\end{exa}
\noindent {\bf Acknowledgment.} We are so grateful to the referee for his/her great effort in proof reading the manuscript.


\begin{thebibliography}{999}

\bibitem{3} D. D. Anderson and M. Naseer, {\em Beck’s coloring of a commutative ring}, J. Algebra 159 (1993), 500--514.

\bibitem{4}  D. F. Anderson and A. Badawi, {\em On the zero-divisor graph of a ring}, Comm. Algebra  36 (2008), 3073--3092.

\bibitem{AB2} D. F. Anderson and A. Badawi, {\em The total graph of a commutative ring}, J. Algebra 320 (2008), 2706--2719.

\bibitem{AB3} D. F. Anderson and A. Badawi, {\em The total graph of a commutative ring without the zero element}, J. Algebra Appl. (2012)  doi: 10.1142/S0219498812500740.

\bibitem{AB5} D. F. Anderson and A. Badawi, {\em The generalized total graph of a commutative ring}, J. Algebra Appl. (2013) doi: 10.1142/S021949881250212X.

\bibitem{ALJ}  D. F. Anderson and J. D LaGrange, {\em Some remarks on the compressed zero-divisor graph}, J. Algebra 447 (2016), 297--321.

\bibitem{ALJ2} D. F. Anderson and J. D LaGrange, {\em The semilattice of annihilator classes in a reduced commutative ring}, Comm. Algebra 43 (2015), 29--42.

\bibitem{AL} D. F. Anderson and P. S. Livingston, {\em The zero-divisor graph of a commutative ring}, J. Algebra 217 (1999), 434--447.


\bibitem{AM}  D. F. Anderson and S. B. Mulay, {\em On the diameter and girth of a zero-divisor graph}, J. Pure Appl. Algebra 210 (2007), 543--550.  
\bibitem{AW} D. F. Anderson and D. Weber, {\em The zero-divisor graph of a commutative ring without identity}, Int. Electron. J. Algebra 23 (2018), 176--202.
    
\bibitem{AL2} D. F. Anderson and E. F. Lewis, {\em A general theory of zero-divisor graphs over a commutative ring}, Int. Electron. J. Algebra 20 (2016), 111--135.

\bibitem{AS} S. Akbari, H. R. Maimani and S. Yassemi,{\em When a zero-divisor graph is planar or a complete r-partite graph}, J. Algebra 270 (2003), 169--180.

\bibitem{AK2} S. Akbari, D. Kiani, F. Mohammadi and S. Moradi, {\em The total graph and regular graph of a commutative ring}, J. Pure Appl. Algebra 213 (2009), 2224--2228.

\bibitem{A1} A. Badawi, {\em On the dot product graph of a commutative ring}, Comm. Algebra 43 (2015), 43--50.

\bibitem{bad2} A. Badawi,{\em On the annihilator graph of a commutative ring}, Comm. Algebra 42 (2014), 108--121.

\bibitem{B}  I. Beck, {\em Coloring of commutative rings}, J. Algebra 116 (1988), 208--226.

\bibitem{KL} F. C. Kimball and J. D. LaGrange, {\em The idempotent-divisor graphs of a commutative ring}, Comm. Algebra 46 (2018), 3899--3912.  

\bibitem{29} S. B.  Mulay, {\em Cycles and symmetries of zero-divisors}, Comm. Algebra 30 (2002), 3533--3558.

\bibitem{45} R. Nikandish, M. J. Nikmehr and M. Bakhtyiari, {\em Coloring of the annihilator graph of a commutative ring},  J. Algebra Appl. 15 (2016), doi: 10.1142/S0219498816501243
   
   \bibitem{LA} J. D. LaGrange, {\em The x-divisor pseudographs of a commutative groupoid}, Int. Electron. J. Algebra 22 (2017), 62--77.

\bibitem{WL} W. L.  LeVeque, Fundamentals of number theory, Addison-Wesley Publishing Company, Reading, Massachusetts, 1977.

\bibitem{radius} Z. Pucanovi\'{c} and Z. Petrovi\'{c}, {\em On the radius and the relation between the total graph of a commutative ring and its extensions}, Publ. Inst. Math. (Beograd) (N.S.) 89 (2011), 1--9.

 \bibitem{SB}  P. K. Sharma and S. M. Bhatwadekar, {\em A note on graphical representations of rings}, J. Algebra  176 (1005), 124--127.

\bibitem{C1} M . Sivagami and T. Tamizh Chelvam, {\em On the trace graph of matrices}, Acta Math. Hungar, 158 (2019), 235--250.

\bibitem{WW}  S. Spiroff and C. Wickham, {\em A zero divisor graph determined by equivalence classes of zero divisors}, Comm. Algebra 39 (2011), 2338--2348.

\bibitem{C2} T. Tamizh Chelvam and T. Asir, {\em  On the genus of the total graph of a commutative ring}, Comm. Algebra 41 (2013), 142--153. 



\end{thebibliography}
\end{document}